\documentclass[letterpaper, 10 pt, conference]{ieeeconf}

\IEEEoverridecommandlockouts      % This command is only
\usepackage{epsfig}
\usepackage{epstopdf}
\usepackage{cite}
\usepackage{graphicx}
\graphicspath{{Figures/}}
\usepackage{url}
\usepackage{stfloats}
\usepackage{amsmath}
\usepackage{amsfonts}   % to get mathbb for the reel symbol
\usepackage{amsopn}     % to get \DeclareMathOperator
\usepackage{psfrag}
\usepackage[tight]{subfigure}
\usepackage{mdwtab}     %ERW: added to get pretty tables
\usepackage{enumerate}
\usepackage{color}

\newtheorem{theorem}{\it Theorem}
\newtheorem{prop}{\it Proposition}

\newtheorem{lemma}{\it Lemma}

\newtheorem{rem}{\it Remark}

\begin{document}

\title{\LARGE \bf Generation of and Switching among Limit-Cycle Bipedal Walking Gaits}
%: Smooth Modulation of Virtual Constraints

\author{Sushant~Veer, Mohamad~Shafiee~Motahar, and Ioannis~Poulakakis
\thanks{S. Veer, M. S. Motahar and I. Poulakakis are with the Department of Mechanical Engineering, University of Delaware, Newark, DE 19716, USA;  e-mail: {\tt\small \{veer, motahar, poulakas\}@udel.edu.}}
\thanks{This work is supported in part by NSF CAREER Award IIS-1350721 and by NRI-1327614.}
}

\maketitle
%%%%%%%%%%%%%%%%%%%%%%
%% Remove Before Submission
%%%%%%%%%%%%%%%%%%%%%%
%\thispagestyle{plain}   %adds page numbering
%\pagestyle{plain}

%===================================================================
%===================================================================
\begin{abstract}
%===================================================================
%===================================================================
In this paper we provide a method to generate a continuum of limit cycles using a single precomputed exponentially stable limit cycle designed within the Hybrid Zero Dynamics framework. Guarantees for existence and stability of these limit cycles are provided. We derive analytical constraints that ensure boundedness of the state under arbitrary switching among a finite set of limit cycles extracted from the continuum. These limit cycles are used for changing the speeds of an underactuated planar bipedal model while satisfying all modeling constraints such as saturation torque and coefficient of friction in a provably correct manner. A strongly connected directed graph of allowable limit cycle switches is built to obtain valid limit cycle transitions for speed changes within 0.42-0.81 m/s.
\end{abstract}

\IEEEpeerreviewmaketitle

%================================================================
%================================================================
\section{Introduction}
\label{sec:intro}
%================================================================
%================================================================

A single limit-cycle gait of a bipedal walker encodes certain characteristic attributes, like average speed or toe clearance. When evolving in the neighborhood of such gait, the biped cannot deviate substantially from these nominal attributes. Thus, increasing the richness of the behaviors exhibited by a dynamically walking bipedal robot entails the generation of multiple limit cycles so that stable switching among them can be realized. However, this can be a challenging task, for the generation of each gait typically  involves numerical integration of the high-dimensional nonlinear dynamics of the system. Ensuring stable switching on the other hand requires estimating the basin of attraction (BoA) of each individual gait. This paper proposes an analytical method for generating a continuum of limit cycles for underactuated dynamic bipeds, while providing guarantees of boundedness of the state under switching among them.
% by leveraging the analytical nature of Hybrid Zero Dynamics (HZD) \cite{westervelt2003hybrid}.

Quasi-static bipedal walkers are capable of a rich variety of behaviors as documented in \cite{Harada2010}. On the other hand, dynamically walking bipeds have not enjoyed similar success, primarily due to the difficulties associated with stabilizing their motions. In the context of underactuated limit-cycle walkers, there have been various efforts toward generating stable and robust gaits; see \cite{freidovich2009passive, gregg2009reduction, westervelt2003hybrid}. Recently, control Lyapunov Function (CLF) based methods were used to enhance the capabilities of bipedal robots~\cite{ames2014rapidly}, including foot placement planning as in \cite{nguyen2015safety} for example. Limit cycles robust to rough terrain disturbances were designed in \cite{hamed2016exponentially}. A method that can be used to expand the BoA of limit cycles was presented in \cite{tedrake2010}. Speed adaptation of HZD (Hybrid Zero Dynamics) based walkers in the presence of an external force \cite{veer2015adaptation} was exploited to realize collaborative object transportation in \cite{motahar2015impedance}. It must be noted that all the above papers either focus on the search for a better limit cycle or on enhancing the properties of an existing one.
%, a variety of methods have been proposed.
%Structure - continuum, but need to switch to expand capabilities, switching requires boa estimation which becomes challenging as the dimension increases, thus most papers lack guarantees on switching or use full actuation, an exception to this is our CDC paper. In this paper we propose a method that guarantees stable switching between limit cycles by checking an analytical condition without requiring the estimation of boa.

Efficient online/offline generation of limit cycles has gained considerable interest recently. A continuum of limit cycles was generated in \cite{razavi2016symmetry} by systematically exploiting the symmetry in idealized walking models. An online gait generation method using nonlinear programming solvers was presented in \cite{hereid2016online}.
To exploit the availability of these limit cycles, the ability to switch among them is required. Switching among a continuum of limit cycles was performed in \cite{da20162d, quan2016dynamic}, while stochastic and supervised learning based switching policies were presented in \cite{saglam2013switching} and \cite{da2016first}, respectively. The aforementioned papers do not provide guarantees of stability under switching. An exception to this is the authors' previous work \cite{motahar2016composing} where provably stable switching was employed to plan motions of a 3D biped in an environment cluttered by obstacles. 
Perhaps, the most challenging aspect of providing such guarantees is obtaining estimates of the BoA of the limit cycles involved. Here, we propose a method that ensures boundedness of the state while switching among multiple limit cycles, without requiring the estimation of the BoA.

This paper proposes an analytical approach for generating a continuum of exponentially stable limit cycles for HZD based bipeds, while providing guarantees for switching among them in the absence of external disturbances. Motivated by \cite{3D-robotica2012} that induces turning in a 3D bipedal robot, the proposed method generates a continuum of limit cycles by smoothly modulating the virtual holonomic constraints that determine the biped's gait. The choice of outputs ensures hybrid invariance of the zero dynamics manifold despite switching among limit cycles, thereby greatly facilitating the commute from one limit cycle to another. A graph of feasible limit cycle switches that comply with modeling constraints is constructed to realize speed planning for a bipedal robot model. The proposed method is easy to implement and retains the analytical tractability associated with the HZD \cite{westervelt2003hybrid}. This paper provides a step towards enriching the behaviors exhibited by dynamic bipedal walkers so that they can accomplish tasks that require greater locomotion flexibility than that provided by a single limit-cycle walking gait.

%The paper is organized as follows. Section~\ref{sec:mod_con} presents the system being studied and the controllers used. Section~\ref{sec:stability} discusses the existence and stability of the generated limit cycles. Section~\ref{sec:switch} provides analytical conditions that ensure stable switching between limit cycles. Section~\ref{sec:speed} particularizes the method for the application of speed change and Section~\ref{sec:conclusion} concludes the paper. 

%================================================================
%================================================================
\section{Modeling and Control}
\label{sec:mod_con}
%================================================================
%================================================================

We study bipedal robots with a single degree of underactuation such as RABBIT \cite[Table~I]{westervelt2003hybrid}; see Fig.~\ref{fig:model}. The robot model has two legs with knees and a torso. The contact between the stance foot and the ground is modeled as unactuated pivot joint while the rest of the robot's joints---two at the knees and two at the hip---are actuated. 

%fig
\begin{figure}[t]
%\vspace{-0.25in}
\begin{centering}
\includegraphics[width=0.5\columnwidth]{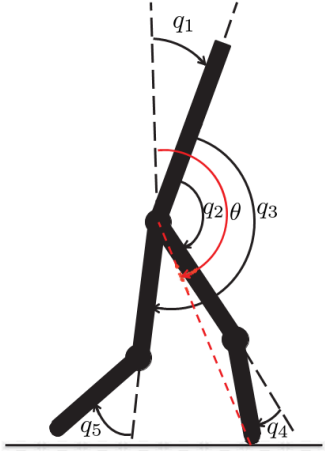} 
\par\end{centering}
\vspace{-0.1in}
\caption{Robot model with a choice of generalized coordinates.}
\vspace{-0.1in}
\label{fig:model} 
\end{figure}
%fig 

\subsection{Model}
The configuration space $Q$ is a subset of the physically realizable configurations of the robot and let $q:=(q_1,q_2,q_3,q_4,q_5)$ be a set of coordinates on $Q$. The dynamics of the swing phase is
%eq
\begin{equation}\label{eq:swing-dyn-q}
D(q) \ddot{q} + C(q,\dot{q})\dot{q} + G(q) = Bu \enspace, 
\end{equation}
%eq
where $D$, $C$ are the inertia and Coriolis matrices, $G$ is the gravitational vector, and $B$ maps the actuator inputs to the generalized forces. In state-space form, \eqref{eq:swing-dyn-q} becomes 
%eq
\begin{equation}\label{eq:swing-dyn-x}
\dot{x} = f(x) + g(x) u \enspace.
\end{equation}
%eq
where $x := (q', \dot{q}')' \in TQ := \{ (q',\dot{q}')' ~|~ q\in Q,~\dot{q}\in \mathbb{R}^5 \}$. The swing phase terminates in an instantaneous double support phase when the swing leg hits the ground. The set of states for which a valid foot impact occurs is called the switching surface and is defined as
%eq
\begin{equation}\label{eq:switching-surface}
\mathcal{S} := \{ (q,\dot{q}) \in TQ ~|~ p_{\rm v}(q) = 0,~\dot{p}_{\rm v}(q,\dot{q})<0 \} \enspace,
\end{equation}
%eq
where $p_{\rm v}(q)$ represents the height of the swing foot. The impact of the foot with the ground reinitalizes the swing phase according to the impact map $\Delta:\mathcal{S} \to TQ$ which maps the states before impact $x^-$ to those after impact $x^+$,
%eq
\begin{equation}\label{eq:Delta}
x^+ = \Delta(x^-) \enspace,
\end{equation}
%eq
see \cite{westervelt2003hybrid} for more details. This gives rise to a hybrid system that has alternating swing and double support phases
%eq
\begin{eqnarray}\nonumber 
  \Sigma_{\rm o}:
  \begin{cases}
    \begin{aligned}
        \dot{x} &= f(x) + g(x) u , & x\notin\mathcal{S} \\
        x^+& = \Delta(x^-), & x\in\mathcal{S}
    \end{aligned}
  \end{cases}.
\end{eqnarray}
%eq

\subsection{Controller}
The controller used in this paper is designed within the HZD framework. The controlled joints are $q_{\rm a}:=(q_2, q_3, q_4, q_5)^{\rm T}$. The following output is associated to \eqref{eq:swing-dyn-x},
%eq
\begin{equation}\label{eq:hd-output}
y = h(q) := q_{\rm a} - h_{\rm d} \circ \theta(q) \enspace,
\end{equation}
%eq
where $\theta(q):=q_1 + q_2 + \frac{1}{2}q_4$ is shown in Fig.~\ref{fig:model}. We restrict our attention to gaits in which $\theta(q)$ increases monotonically during the step. The nominal control law $u^*(x) = -L_gL_fh^{-1}(x)L_f^2h(x)$ renders the zero dynamics surface 
%eq
\begin{equation}\label{eq:nom-Z}
\mathcal{Z} := \{ (q,\dot{q})\in TQ~|~h(q)=0,~L_f h(q,\dot{q})=0 \} \enspace,
\end{equation}
%eq
invariant under the closed-loop swing dynamics. Additionally, the design of $h_{\rm d}(\theta)$ according to~\cite[Section~V.A]{westervelt2003hybrid}, ensures the invariance of \eqref{eq:nom-Z} under impact \eqref{eq:Delta}. 

Our method of generating a continuum of limit cycles relies on suitably modifying the output function \eqref{eq:hd-output} to include an additional term $h_{\rm s}$ as follows
%eq
\begin{equation}\label{eq:beta-output}
y_\beta = h_\beta(q) = q_{\rm a} - h_{\rm d} (\theta) - h_{\rm s}(\theta,\beta)\enspace.
\end{equation}
%eq
The term $h_{\rm s}(\theta,\beta)$ is a polynomial of $\theta$ with coefficients dependent on the parameters $\beta \in \mathbb{R}^{\dim (\beta)}$, and it is designed as follows. First, we require that
%eq
\begin{equation}\label{eq:hs-conditions-1}
h_{\rm s}(\theta,0)=0,~~~~\mathrm{for}~\theta^-\leq\theta\leq\theta^+ \enspace,
\end{equation}
%eq
so that when $\beta=0$, the modified output \eqref{eq:beta-output} reduces to the original output \eqref{eq:hd-output}; i.e., $h_\beta(q) = h(q)$. Then, we define $\theta_{\rm s} = \theta^+ + 0.9 (\theta^- - \theta^+)$, where $\theta^+$ and $\theta^-$  are the values of $\theta$ at the beginning (post-impact) and the end (pre-impact) of a step, and we impose the following conditions
%eq
\begin{eqnarray}\label{eq:hs-conditions}
	\begin{cases}
		\begin{aligned}
		& h_{\rm s}(\theta^+,\beta) = 0,~\frac{\partial h_{\rm s}}{\partial \theta}(\theta^+,\beta) = 0\\
		& h_{\rm s}(\theta_s,\beta) = 0,~\frac{\partial^i h_{\rm s}}{\partial \theta^i}(\theta_s,\beta) = 0,~i = 1,2\\
		& h_{\rm s}(\theta,\beta) = 0,~~\mathrm{for}~~\theta_s \leq \theta \leq \theta^-
		\end{aligned}
	\end{cases}
\end{eqnarray}
%eq
Essentially, the polynomials $h_{\rm s}(\theta,\beta)$ vanish at the post-impact instant (when $\theta = \theta^+$) and after 90\% of the step is completed (when $\theta \in [\theta_{\rm s}, \theta^-)$). Similar output designs were proposed in \cite{3D-robotica2012} to induce turning in a 3D biped; here, we show that by merely picking different $\beta$, we can smoothly modulate the output and generate limit cycles corresponding to different walking gaits. 
%An example of $h_{\rm s}(\theta)$ can be seen in Fig.~\ref{fig:hs}. 
%%fig
%\begin{figure}[b]
%%\vspace{-0.25in}
%\begin{centering}
%\includegraphics[width=1.0\columnwidth]{figureEPS/hs.eps} 
%\par\end{centering}
%\vspace{-0.15in}
%\caption{An illustration of $h_{\rm s}(\theta)$.}
%\vspace{-0.1in}
%\label{fig:hs} 
%\end{figure}
%%fig

The zero dynamics surface associated with $h_\beta$ is 
%eq
\begin{equation}\label{eq:prim-Z}
\mathcal{Z}_\beta := \{ (q,\dot{q})\in TQ~|~h_\beta(q)=0,~L_f h_\beta(q,\dot{q})=0 \}\enspace,
\end{equation}
%eq
and is rendered invariant under the \eqref{eq:swing-dyn-x} in closed loop with the control law $u_\beta^*(x) = -L_{g}L_{f} h_\beta(x)^{-1}L^2_f h_\beta(x) $. The resulting closed-loop system is 
%. Implementing $u_\beta^*(x)$ to the open loop system $\Sigma_{\rm o}$ results in 
%eq
\begin{eqnarray}\nonumber 
  \Sigma:
  \begin{cases}
    \begin{aligned}
        \dot{x} &= f(x) + g(x) u_\beta^*(x) , & x\notin\mathcal{S} \\
        x^+& = \Delta(x^-), & x\in\mathcal{S}
    \end{aligned}
  \end{cases}.
\end{eqnarray}
%eq

Finally, $\mathcal{Z}_\beta$ can be rendered attractive by modifying $u_\beta^*(x)$ as $u_\beta(x) = u_\beta^*(x) + L_{g}L_{f} h_\beta(x)^{-1} \nu$, where $\nu$ is an auxiliary control term. In this paper we chose a CLF based control law to generate $\nu$ using a Quadratic Program (QP) with constraints on the motor torques, the coefficient of friction, and the unilateral ground reaction force so that the physical limitations of the model are respected; see \cite[Section~4.1]{quan2016dynamic} for details about the CLF based QP.

The following lemma establishes some useful properties of $\mathcal{Z}_\beta$ and will be invoked frequently throughout the paper.
%lem
\begin{lemma}\label{lem:Z-properties}
Let $h_{\rm s}(\theta,\beta)$ satisfy \eqref{eq:hs-conditions} and $\mathcal{Z}$, $\mathcal{Z}_\beta$ be defined as in \eqref{eq:nom-Z} and \eqref{eq:prim-Z} respectively. Then, for all $\beta \in \mathbb{R}^{\dim (\beta)}$
\begin{enumerate}[i)]
\item $\mathcal{S}\cap\mathcal{Z}_\beta = \mathcal{S}\cap\mathcal{Z}$, and
\item $\mathcal{Z}_\beta$ is hybrid invariant.
\end{enumerate}
\end{lemma}
%lem
\begin{proof}
The proof of i) follows from the last condition of \eqref{eq:hs-conditions}, which ensures that before the end of the step, i.e., before the state arrives at $\mathcal{S}$, the outputs $h_\beta(q) = h(q)$. 
The proof of ii) follows from the hybrid invariance of $\mathcal{Z}$. Let $x\in\mathcal{S}\cap\mathcal{Z}_\beta\iff x\in\mathcal{S}\cap\mathcal{Z}$ by the first part of Lemma~\ref{lem:Z-properties}. Impact invariance of $\mathcal{Z}$ guarantees that $h(\Delta(x))=0$ and $L_f h(\Delta(x))=0$. Using the first two conditions of \eqref{eq:hs-conditions} we have $h_\beta(\Delta(x)) = h(\Delta(x))=0$ and  $L_f h_\beta(\Delta(x))=L_f h(\Delta(x))=0$ implying impact invariance for $\mathcal{Z}_\beta$. Invariance of $\mathcal{Z}_\beta$ under continuous dynamics is ensured by the choice of $u_\beta^*(x)$.
\end{proof}

The result of Lemma~\ref{lem:Z-properties} can be geometrically illustrated in Fig.~\ref{fig:manifold}. It can be seen that the modified output \eqref{eq:beta-output} smoothly deforms the zero dynamics surface $\mathcal{Z}$ associated with the original output \eqref{eq:hd-output}, but it does so in a way that the deformed surface $\mathcal{Z}_\beta$ coincides with $\mathcal{Z}$ at the beginning of the step (i.e., at $\Delta(\mathcal{S}\cap\mathcal{Z})$) and after 90\% of the step is completed.
%Similarly, $\mathcal{Z}_\beta$ converges back to $\mathcal{Z}$ before it approaches $\mathcal{S}\cap\mathcal{Z}$. This geometric structure for $\mathcal{Z}_\beta$ is a consequence of the constraints \eqref{eq:hs-conditions} on $h_{\rm s}(\theta)$.

%fig
\begin{figure}[h!]
%\vspace{-0.25in}
\begin{centering}
\includegraphics[width=1.0\columnwidth]{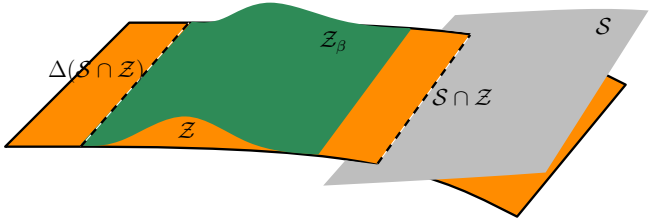} 
\par\end{centering}
\vspace{-0.1in}
\caption{Geometric illustration of $\mathcal{Z}$, $\mathcal{Z}_\beta$, and $\mathcal{S}$. The zero dynamic surface $\mathcal{Z}$ associated with $h(\theta)$ is orange, the zero dynamic surface $\mathcal{Z}_\beta$ associated with $h_\beta(\theta)$ is green, and the switching surface $\mathcal{S}$ is grey. The dashed lines represent $\mathcal{S}\cap\mathcal{Z}$ and $\Delta(\mathcal{S}\cap\mathcal{Z})$}.
\vspace{-0.1in}
\label{fig:manifold} 
\end{figure}
%fig 

Before continuing with using the constructions of this section to generate a continuum of limit cycles, the following remark states that the stride length corresponding to the ``base'' limit cycle obtained for $\beta=0$ remains the same for all the limit cycles generated for $\beta \neq 0$.  
%rem
\begin{rem}\label{rem:theta+-same}
An outcome of Lemma~\ref{lem:Z-properties}i) is that $\theta^+$ and $\theta^-$ do not depend on $\beta$. Indeed, according to \cite[HH5)]{westervelt2003hybrid}, there exists a unique configuration $q^-$ where the state reaches $\mathcal{S}\cap\mathcal{Z}$. Since by Lemma~\ref{lem:Z-properties}i) $\mathcal{S}\cap\mathcal{Z}_\beta=\mathcal{S}\cap\mathcal{Z}$, it must be that $\theta^- = \theta(q^-)$ independent of $\beta$. Furthermore, the configuration $q$ remains unaltered through the impact; only the swing and stance legs switch roles. Hence, $q^-$ being independent of $\beta$ implies the same for $q^+$, which further leads to $\theta^+ = \theta(q^+)$ being the same regardless of $\beta$.
\end{rem}
%rem

%================================================================
%================================================================
\section{Continuum of Limit Cycles}
\label{sec:stability}
%================================================================
%================================================================

In this section we prove the existence of a continuum of limit cycles and study their stability properties using the method of Poincar\'e. Let $\varphi_\beta(t,x)$ be the maximal solution for the continuous dynamics of $\Sigma$. The time-to-impact function $T_I : TQ \times \mathbb{R}^{\mathrm{dim}(\beta)} \rightarrow \mathbb{R}_+$ can be defined as 
%eq
\begin{equation} \label{eq:time_to_impact}
    T_I(x,\beta) \!:=\! 
        \inf \left\{ t \geq t_k ~|~ \varphi_\beta(t,\Delta(x)) \in \mathcal{S}\right\}.
        %\in [t_k,+\infty)
\end{equation}
%eq
where $k\in\mathbb{Z}_+$ denotes the step number and $t_k$ is the starting time of the $k$-th step. The Poincar\'e map $P:\mathcal{S}\times \mathbb{R}^{\mathrm{dim}(\beta)} \to \mathcal{S}$  is defined as
%eq
\begin{equation}\label{eq:poincare-full}
P(x,\beta):=\varphi_\beta(T_I(x,\beta),\Delta(x)) \enspace.
\end{equation}
%eq
In what follows, we assume that there exists a limit cycle for $\beta=0$, hence, $P(x^*,0)=x^*$. This is the ``base'' limit cycle.

%prop
\begin{prop}[Existence of Limit Cycles]\label{prop:existence-fp}
Let $(x^*,0)$ be a fixed point of $P(x,\beta)$ defined in \eqref{eq:poincare-full}. If $\frac{\partial P}{\partial x}|_{(x^*,0)}$ does not have an eigenvalue at 1, there exists a $\delta>0$ such that for each $\beta \in B_\delta(0) \subset \mathbb{R}^{\dim (\beta)}$, there is a unique solution of $P(x,\beta)=x$ given by $x = \mathrm{x}^*(\beta)$. Further, $\mathrm{x}^*(\beta)$ is continuous for $\beta\in B_\delta(0)$.
\end{prop}
%prop
%proof
\begin{proof}
The proof follows from the implicit function theorem in view of the fact that there exists a fixed point for $\beta=0$. Define $H: {\cal S} \times \mathbb{R}^{\dim (\beta)} \to TQ$ by the rule $H(x,\beta):= P(x,\beta)-x$; the function $H$ is continuously differentiable. We are given that $H(x^*,0)=0$ and $\frac{\partial P}{\partial x}|_{(x^*,0)}$ does not have 1 as an eigenvalue, i.e. $\mathrm{det}(\frac{\partial P}{\partial x}|_{(x^*,0)}-I)=\mathrm{det}(\frac{\partial H}{\partial x}|_{(x^*,0)}) \neq 0$. Then, by \cite[Theorem~2.5]{hartman1964ordinary}, there exists $\delta>0$ for which the statement of Proposition~\ref{prop:existence-fp} holds.
%the implicit function theorem
\end{proof}
%proof

Proposition~\ref{prop:existence-fp} ensures that an infinite number of limit cycles can be generated using a ``base''' limit cycle corresponding $\beta=0$, provided that 1 is not an eigenvalue of $\frac{\partial P}{\partial x}|_{(x^*,0)}$. It should be mentioned that for periodic orbits of smooth systems, the linearization of the Poincar\'e map at the  corresponding fixed point trivially possesses an eigenvalue at 1. However, as noted in~\cite[Section 3]{wendel2012rank}, this property does \emph{not} necessarily hold for periodic orbits of hybrid systems. In such systems, the trivial eigenvalue---i.e., the one associated with the Poincar\'e reduction---need not be located at 1; see\cite[Theorem 3]{wendel2012rank}. This is in fact the case for the Poincar\'e map \eqref{eq:poincare-full}, which does not possess any eigenvalue at 1.
% due to the fact that limit cycles for such systems may not be closed
%It must be clarified that for hybrid systems such as the one under study in this paper, linearization of the Poincar\'e map does not necessarily result in an eigen-value 1 unlike Poincar\'e maps for limit cycles of continuous systems. Instead, perturbations out of the switching surface $\mathcal{S}$ yield an eigen-value 0 \cite{wendel2012rank}.

Now we turn our attention towards the stability of these limit cycles. As in \cite[Theorem~1]{westervelt2003hybrid}, let $\xi = (\theta,\zeta)$ be coordinates on $\mathcal{Z}_\beta$, where $\zeta := \frac{1}{2}(D_1(q)\dot{q})^2$ and $D_1(q)$ is the first row of the inertia matrix $D$ in \eqref{eq:swing-dyn-q}. With the knowledge that $\mathcal{Z}_\beta$ is hybrid invariant from Lemma~\ref{lem:Z-properties}ii), using \cite[Theorem~3]{westervelt2003hybrid} we have that the reduced Poincar\'e map $\rho_\beta:=P(x,\beta)|_{\mathcal{Z_\beta}} : \mathcal{S}\cap\mathcal{Z} \to \mathcal{S}\cap\mathcal{Z}$ takes the form
% of the hybrid zero dynamics
%eq
\begin{equation}\label{eq:reduced-poinc-map}
\rho_\beta(\zeta) := \delta_{\mathrm{z},\beta}^2 \zeta - V_\beta(\theta^-) \enspace,
\end{equation}
%eq
where $\delta_{\mathrm{z},\beta}$ and $V_\beta(\theta^-)$ are constants\footnote{In fact, $V_\beta$ is a function of $\theta$ but in $\rho_\beta$ we only need its value at $\theta^-$.}. The reduced Poincar\'e map gives rise to a discrete dynamical system
%eq
\begin{equation}\label{eq:rho-system}
\zeta[k+1] = \rho_\beta(\zeta[k])\enspace,
\end{equation}
%eq
where $k\in\mathbb{Z}_+$. The fixed point of \eqref{eq:rho-system} given by
%eq
\begin{equation}\label{eq:fp}
\zeta_\beta^*=-\frac{V_\beta(\theta^-)}{1-\delta_{\mathrm{z},\beta}^2}
\end{equation} 
%eq
is exponentially stable if, and only if, $\delta_{\mathrm{z},\beta}<1$.

The following result establishes conditions under which the limit cycles generated from a locally exponentially stable base limit cycle are themselves locally exponentially stable.
%thm
\begin{theorem}[Stability of Limit Cycles]\label{thm:stability}
Let $\rho_\beta(\zeta)$ be the reduced Poincar\'e map as defined in \eqref{eq:reduced-poinc-map}. Suppose that $\zeta^*:=\zeta_0^*$ is an exponentially stable fixed point of $\rho(\zeta):=\rho_0(\zeta)$ corresponding to $\beta=0$. Then, for $B_\delta(0)$ established in Proposition~\ref{prop:existence-fp} and any $\beta \in B_\delta(0)$, $\zeta_\beta^*$ is an exponentially stable fixed point of \eqref{eq:rho-system}.
\end{theorem}
%thm
%proof
\begin{proof}
In \cite[Section~IV.A]{westervelt2003hybrid}, it is shown that $\delta_{\mathrm{z},\beta}$ depends on the impact configuration of the robot $q^-$ and $\frac{\partial h_\beta}{\partial q}(q^-)$. By Remark~\ref{rem:theta+-same}, $q^-$ is the same for all $\beta$ and by Lemma~\ref{lem:Z-properties}i), $\frac{\partial h_\beta}{\partial q}(q^-)=\frac{\partial h}{\partial q}(q^-)$. Thus, for all $\beta\in B_\delta(0)$, we have\footnote{As $\delta_{\rm z,\beta}$ is the same for all $\beta\in B_\delta(0)$, we use $\delta_{\rm z}:=\delta_{\mathrm{z},0}$ from hereon.} $\delta_{\mathrm{z},\beta}=\delta_{\rm z}<1$, establishing the result. 
\end{proof}
%proof
The exponential stability of $\rho_\beta(\zeta)$ can be lifted to local exponential stability of the full-order Poincar\'e map $P(x,\beta)$ by choosing a fast enough convergence rate for the output dynamics; this can be achieved through the CLF based design of $\nu$ as in~\cite[Theorem~2]{ames2014rapidly}. 
The existence of a continuum of locally exponentially stable limit cycles is helpful in extending the richness of the biped's behaviors by switching among these elementary limit cycles. However, switching must ensure that the stability and modeling constraints are not violated. In what follows, we focus on guaranteeing that the biped remains well behaved under switching, provided that there are no external disturbances exciting dynamics outside of the zero dynamics surface.
%%Theorem~\ref{thm:stability} guarantees that if the original limit cycle is exponentially stable, all other limit cycles generated using this method will be exponentially stable as well. 

%================================================================
%================================================================
\section{Switching Among Limit Cycles}
\label{sec:switch}
%================================================================
%================================================================

We consider a finite set $\mathbb{B} = \{ \beta_p, p\in\mathcal{P} \}$ of parameter arrays $\beta$, indexed by $p \in \mathcal{P}$. Each parameter array $\beta_p$  corresponds to an exponentially stable limit cycle computed using the output modulation method presented in Section~\ref{sec:mod_con}. Let $\sigma:\mathbb{Z}_+ \to \mathcal{P}$ be a switching signal mapping the step number $k$ to the index $p=\sigma(k)$ of the parameters $\beta_p$ that characterize the controller applied at that step. This gives rise to a discrete switched system of the form\footnote{With an abuse of notation from hereon we will use $p\in\mathcal{P}$ as a subscript instead of $\beta$. Also we use $P_p(x)$ instead of $P(x,\beta_p)$.}
%eq
\begin{equation}\label{eq:switch-system-full}
x[k+1] = P_{\sigma(k)}(x[k]) \enspace.
\end{equation}
%eq
Note that the switched system \eqref{eq:switch-system-full} differs from those in \cite{liberzon2003switching} in that the individual systems do not share a common equilibrium. 
Obtaining conditions on the switching signal which guarantee that the full-order dynamics \eqref{eq:switch-system-full} remains well behaved is computationally prohibitive; although~\cite[Theorem 1]{motahar2016composing} can be used, it  essentially requires estimating the basin of attraction of the fixed points of \eqref{eq:switch-system-full} in the full order system. However, for a class of practical applications, where a higher-level logic governs switching---as in motion planning amidst obstacles~\cite{motahar2016composing} or supervisory adaptive control~\cite{veer2017supervisory}---the analysis can be simplified by considering the case where no external disturbances are present.
In this case, the following result ensures that if \eqref{eq:switch-system-full} is initialized on $\mathcal{S}\cap\mathcal{Z}$, then the state always evolves on the corresponding $\mathcal{Z}_p$, despite the presence of switching among $p \in \mathcal{P}$. 
%among $p \in \mathcal{P}$, the state always evolves on the corresponding $\mathcal{Z}_p$. 
%This statement is made mathematically precise in the following proposition.}
%Hence, we cannot directly apply classical switched systems theory here. 

%prop
\begin{prop}\label{prop:switch-zd}
Let $\sigma$ be the switching signal applied on \eqref{eq:switch-system-full}. If $x[0]\in\mathcal{S}\cap\mathcal{Z}$, then for any $k\in\mathbb{Z}_+$, the following holds: $\varphi_{\sigma(k+1)}(t,\Delta(x[k])) \in \mathcal{Z}_{\sigma(k+1)}$ for $t\in[t_k,T_I(x,\beta_{\sigma(k+1)}))$.
\end{prop}
%prop
%proof
\begin{proof}
The poof is by induction. The induction begins at $k=0$ where $x[0]\in\mathcal{S}\cap\mathcal{Z}$. Let $x[k]\in\mathcal{S}\cap\mathcal{Z} = \mathcal{S}\cap\mathcal{Z}_{\sigma(k)}$ by Lemma~\ref{lem:Z-properties}i). By Lemma~\ref{lem:Z-properties}ii) we have that $\mathcal{Z}_{\sigma(k+1)}$ is hybrid invariant, hence it is also impact invariant. Thus, $\Delta(\mathcal{S}\cap\mathcal{Z}_{\sigma(k+1)})\subset \mathcal{Z}_{\sigma(k+1)}$. However, by Lemma~\ref{lem:Z-properties}i) we also have $\Delta(\mathcal{S}\cap\mathcal{Z}_{\sigma(k)})=\Delta(\mathcal{S}\cap\mathcal{Z}_{\sigma(k+1)})$. Thus, we obtain $\Delta(\mathcal{S}\cap\mathcal{Z}_{\sigma(k)})\subset \mathcal{Z}_{\sigma(k+1)}$ which means that $\Delta(x[k])\in\mathcal{Z}_{\sigma(k+1)}$. The choice of the control law $u_{\sigma(k+1)}^*$ ensures invariance of $\mathcal{Z}_{\sigma(k+1)}$ under continuous dynamics, i.e. $\varphi_{\sigma(k+1)}(t,\Delta(x[k])) \in \mathcal{Z}_{\sigma(k)}$ for $t\in[t_k,T_I(x,\beta_{\sigma(k)}))$. Hence, on the next return to $\mathcal{S}$, we have $x[k+1]\in\mathcal{S}\cap\mathcal{Z}_{\sigma(k+1)}=\mathcal{S}\cap\mathcal{Z}$, thus completing the proof.
%
%The induction begins at $k=0$ where $x[0]\in\mathcal{S}\cap\mathcal{Z}$. Let $x[k]\in\mathcal{S}\cap\mathcal{Z} \iff x[k]\in\mathcal{S}\cap\mathcal{Z}_{\sigma(k)}$ by Lemma~\ref{lem:Z-properties}i). By Lemma~\ref{lem:Z-properties}ii) we have that $\mathcal{Z}_{\sigma(k+1)}$ is hybrid invariant, hence it is also impact invariant. Thus, $\Delta(\mathcal{S}\cap\mathcal{Z}_{\sigma(k+1)})\subset \mathcal{Z}_{\sigma(k+1)}$. However, by Lemma~\ref{lem:Z-properties}i) we also have $\Delta(\mathcal{S}\cap\mathcal{Z}_{\sigma(k)})=\Delta(\mathcal{S}\cap\mathcal{Z}_{\sigma(k+1)})$. Thus, we obtain $\Delta(\mathcal{S}\cap\mathcal{Z}_{\sigma(k)})\subset \mathcal{Z}_{\sigma(k+1)}$ which means that $\Delta(x[k])\in\mathcal{Z}_{\sigma(k+1)}$, completing the proof.
\end{proof}
%proof

Proposition~\ref{prop:switch-zd} allows us to restrict our study to the discrete switched system given by the restricted Poincar\'e map
%eq
\begin{equation}\label{eq:switch-system-reduced}
\zeta[k+1] = \rho_{\sigma(k)}(\zeta[k]) \enspace.
\end{equation}
%eq
Next, we present conditions under which the evolution of \eqref{eq:switch-system-reduced} remains bounded under arbitrary switching signals. Before we proceed, a few definitions are in order. Let $\zeta_{\rm lb}^*:= \min_{p\in\mathcal{P}} \zeta^*_p$, and $\zeta_{\rm ub}^*:= \max_{p\in\mathcal{P}} \zeta^*_p$. Let $K_p:= \max_{\theta^+\leq\theta\leq\theta^-} V_p(\theta)$, and $K:= \max_{p\in\mathcal{P}} K_p$. With these definitions we present the following theorem.
%thm
\begin{theorem}[Switching Between Limit Cycles]\label{thm:switching}
Consider \eqref{eq:switch-system-reduced} with a switching signal $\sigma$. Let $\zeta_{\rm lb}^* \geq K/\delta_{\rm z}^2$. Then, for any $\sigma$ the solution of \eqref{eq:switch-system-reduced} satisfies
%eq
\begin{equation}\nonumber
\zeta_{\rm lb}^*\leq\zeta[0]\leq \zeta_{\rm ub}^* \implies \zeta_{\rm lb}^*\leq\zeta[k]\leq \zeta_{\rm ub}^* \enspace,
\end{equation}
%eq
for all $k\in\mathbb{Z}_+$.
\end{theorem}
%thm
%proof
\begin{proof}
%We prove the result by showing that 
The domain of definition of $\rho_p(\zeta)$ is given by $\zeta \geq K_p/\delta_{\rm z}^2$ \cite[Theorem~3]{westervelt2003hybrid}. Thus, $\zeta_{\rm lb}^* \geq K/\delta_{\rm z}^2$ ensures that the fixed point for each $p\in\mathcal{P}$ is in the domain of definition of every other fixed point. Thereby, ensuring that the reduced $\rho_p$ is well defined for any $\zeta\geq\zeta_{\rm lb}^*$ and for any $p\in\mathcal{P}$. 
We prove the boundedness of the state by induction. It is given that $\zeta_{\rm lb}^*\leq\zeta[0]\leq \zeta_{\rm ub}^*$. Let us assume that for some $k\in\mathbb{Z}_+$,  $\zeta_{\rm lb}^*\leq\zeta[k]\leq \zeta_{\rm ub}^*$. Then, $\zeta[k+1] = \delta_{\rm z}^2 \zeta[k] -V_{\sigma(k)}(\theta^-)$, and substituting \eqref{eq:fp} results in 
%eq
\begin{equation}\label{eq:poinc-red-system-fp}
\zeta[k+1] = \delta_{\rm z}^2 \zeta[k]+ (1-\delta_{\rm z}^2) \zeta_{\sigma(k)}^* \enspace.
\end{equation}
%eq
Using $1-\delta_{\rm z}^2>0$, and $\zeta_{\rm lb}^*\leq\zeta_{\sigma(k)}^*\leq \zeta_{\rm ub}^*$ in \eqref{eq:poinc-red-system-fp} we can bound $\zeta[k+1]$ by
%eq
\begin{align}\nonumber
\delta_{\rm z}^2 (\zeta[k]-\zeta_{\rm lb}^*) + \zeta_{\rm lb}^* \leq \zeta[k+1] \leq \delta_{\rm z}^2 (\zeta[k]-\zeta_{\rm ub}^*) + \zeta_{\rm ub}^* \enspace.
\end{align}
%eq
Noting that $\zeta[k]-\zeta_{\rm lb}^*\geq 0 $ and $\zeta[k]-\zeta_{\rm ub}^* \leq 0$ ensures that $\zeta_{\rm lb}^*\leq\zeta[k+1]\leq \zeta_{\rm ub}^*$. Hence, by induction we get that for all $k\in\mathbb{Z}_+$, $\zeta_{\rm lb}^*\leq\zeta[k]\leq \zeta_{\rm ub}^*$.
\end{proof}
%proof

Theorem~\ref{thm:switching} provides conditions under which $\zeta$ remains bounded and within the domain of definition of every fixed point, thereby allowing for arbitrary switches. It should be emphasized however that Theorem~\ref{thm:switching} does not account for modeling constraints like motor saturation torque, friction, and ground reaction force, which may be violated during the transients introduced by switching. Hence, arbitrary switching may not be realistically feasible.
%It is pointed out that guarantees on these constraints cannot be provided even under the use of standard techniques that involve Lyapunov estimates of the basin of attraction.
%despite the satisfaction of Theorem~\ref{thm:switching},  

To address this issue, a directed graph $\mathcal{G}$ of feasible switches---i.e., switches which satisfy Theorem~\ref{thm:switching} and do not violate modeling constraints---can be constructed. The nodes of the graph are fixed points that respect the constraints and the directed edges represent feasible switches. To avoid violation of the modeling constraints, the switching logic needs to wait long enough for the state to get sufficiently close to the destination node before switching occurs. This can be achieved by imposing a dwell-time constraint on the switching signal. The dwell-time $N \in \mathbb{Z}_+$ is the minimum number of steps the biped should take before a switching occurs, i.e. $\sigma(k_i+k)=\sigma(k_i)$ for all $k<N$, where $k_i$ is the discrete-time of the last switch.
The following proposition provides a bound on the dwell-time which guarantees that any state within $\epsilon>0$ of the source fixed point can be commuted within $\epsilon$ of the destination fixed point.
%at the time of switching, the state must be in the vicinity of the fixed point of the source node.} As a result
%; see Fig.~\ref{fig:digraph} for a graph built in Section~\ref{sec:speed}
%from the limit cycle at the tail to the limit cycle at the head of the edge

%In the next section for the particular example of speed change, we build a directed graph of possible switches. It is computationally very expensive to build a graph by discretizing the set $[\zeta_{\rm lb}^*, \zeta_{\rm ub}^*]$ and treating each discrete element as a node. Instead we only build a graph for transitions between limit cycles in $\mathbb{B}$. Thus, to ensure that transients do not break the modeling constraints which are checked to build the graph, we switch only when we are close enough to a fixed point. This requires us to wait for a certain number of steps before we transition to an adjacent node in the graph.

%%rem
%\begin{rem}
%The state $\zeta$ converges within $B_\epsilon(\zeta_{\beta_p}^*)$ after $N_{\rm d}$ steps given 
%\end{rem}
%%rem

%thm
\begin{prop}[Dwell-Time]\label{prop:dwell-time}
Consider \eqref{eq:switch-system-reduced} and assume that the conditions of Theorem~\ref{thm:switching} hold. Let $p, q \in \mathcal{P}$, and assume that for $\epsilon > 0$, $\zeta[k] \in B_\epsilon(\zeta_p^*)$.  Then, for any switching signal $\sigma$ with dwell-time $N_{p \to q} \in \mathbb{Z}_+$ satisfying
%eq
\begin{equation}\label{eq:dwell-time}
N_{p \to q} > \frac{1}{2} \frac{\log\Big(|\zeta_p^*-\zeta_q^*|/\epsilon + 1 \Big)}{\log\big(1/\delta_{\rm z}\big)}
\end{equation}
%eq
we have $\zeta[k+N_{p \to q}] \in B_\epsilon(\zeta_q^*)$.
\end{prop}
%thm
%proof
\begin{proof}
By induction on \eqref{eq:poinc-red-system-fp},
%eq
\begin{equation}\label{eq:n-step-red-poinc}
\zeta[k+n] = \delta_{\rm z}^{2n} (\zeta[k]-\zeta_q^*) + \zeta_q^* \enspace.
\end{equation}
%eq
%To ensure $\zeta[k+N_{\rm d}] \in B_\epsilon(\zeta_q^*)$ we require $N_{\rm d}$ such that
%%eq
%\begin{equation}\label{eq:Nd-1}
%\delta_{\rm z}^{2N_{\rm d}} |(\zeta[k]-\zeta_q^*)|<\epsilon \enspace.
%\end{equation}
%%eq
By triangle inequality and the fact that $\zeta[k]\in B_\epsilon(\zeta_p^*)$
%eq
\begin{align}
|\zeta[k]-\zeta_q^*| & \leq |\zeta[k]-\zeta_p^*| + |\zeta_p^*-\zeta_q^*| \enspace, \nonumber \\
 & < \epsilon + |\zeta_p^*-\zeta_q^*| \enspace. \label{eq:trig-ineq}
\end{align}
%eq
Using \eqref{eq:n-step-red-poinc} and \eqref{eq:trig-ineq} we get
%eq
\begin{align}
|\zeta[k+N_{p \to q}]-\zeta_q^*| & = \delta_{\rm z}^{2N_{p \to q}} |(\zeta[k]-\zeta_q^*)|\enspace, \nonumber \\
& < \delta_{\rm z}^{2N_{p \to q}} \big( \epsilon + |\zeta_p^*-\zeta_q^*| \big) <\epsilon \enspace, \nonumber
\end{align}
%eq
and the last inequality follows from \eqref{eq:dwell-time}.
\end{proof}
%proof

It is important to mention that since $\zeta[k]$ is available through state feedback, switching from $\zeta_p^*$ to $\zeta_q^*$ can be determined by monitoring when $\zeta[k]$ enters the $\epsilon$-ball of the target fixed point, without checking \eqref{eq:dwell-time}. The importance of the availability of the bound \eqref{eq:dwell-time} lies in the fact that it can be used to weight the edges of feasible transitions \emph{a priori}; that is, in the planning stage, before a switching sequence is executed. With the dwell-time bound of Proposition~\ref{prop:dwell-time} determining the weight on the edges of the transition graph, a path can be found that takes the state from the source to the destination node in the least number of steps.
%For traversing from one node to another there maybe multiple routes available.

%=================================================================
%=================================================================
\section{Limit Cycle Switching for Speed Change}
\label{sec:speed}
%=================================================================
%=================================================================

In this section we particularize the theoretical results of Sections~\ref{sec:stability} and~\ref{sec:switch} to obtain provably stable speed planning for the underactuated bipedal robot in Fig.~\ref{fig:model}. 
We begin by designing $h_{\rm d}(\theta)$ in \eqref{eq:hd-output} by following the method in \cite{westervelt2003hybrid} to obtain an exponentially stable periodic gait with an average speed of 0.75 m/s. Then, $h_{\rm s}(\theta,\beta)$ satisfying \eqref{eq:hs-conditions} is augmented to $h_{\rm d}$. The average speed of the biped over a step is represented as a function $v:\mathcal{S}\times \mathbb{R}^{\mathrm{dim}(\beta)}\to \mathbb{R}$. Let $v_{\rm des}$ be the desired speed of the new limit cycle. To generate a limit cycle with an average speed of $v_{\rm des}$ choose $\beta$ as
%eq
\begin{equation}\label{eq:beta-generation}
\beta =  \Big(\frac{\partial v}{\partial x}\Big|_{(x^*,0)}\Big)^\dagger (v_{\rm des}-v(x^*,0)) \enspace,
\end{equation}
%eq
where $\dagger$ represents the right pseudoinverse. As \eqref{eq:beta-generation} is based on linearization, we do not obtain a $\beta$ that realizes $v_{\rm des}$ exactly; however, the sign of $(v_{\rm des}-v(x^*,0))$ holds true in the sense that if $(v_{\rm des}-v(x^*,0))>0$ the average speed of the new limit cycle is faster than $v(x^*,0)$ while if $(v_{\rm des}-v(x^*,0))<0$ the average speed of the new limit cycle is slower than $v(x^*,0)$. Furthermore, it was observed that $v_{\rm des,1}>v_{\rm des,2}$ resulted in limit cycles that satisfy $v(\mathrm{x}^*(\beta_1),\beta_1)>v(\mathrm{x}^*(\beta_2),\beta_2)$.

For a small enough choice of $(v_{\rm des}-v(x^*,0))$, we can ensure that $\beta \in B_\delta(0)$, and hence by Proposition~\ref{prop:existence-fp} we can find limit cycles for each corresponding $(v_{\rm des}-v(x^*,0))$. Using this approach we were able to generate limit cycles anywhere between 0.42 m/s to 0.81 m/s that satisfied the modeling constraints, i.e. saturation torque of 100 Nm, coefficient of friction below 0.8, and a minimum upwards ground reaction force of 100 N. The projection of these limit cycles on the $(\theta,\zeta)$ plane can be seen in Fig.~\ref{fig:continuum}. 
%Limit cycles were discovered below 0.42 m/s as well, but they were 2-step periodic, while limit cycles above 0.81 m/s violated the saturation torque constraint of 100 Nm. We remark that instead of using the above procedure, one can also run a low-dimensional nonlinear optimization on $\beta$ to obtain new limit cycles with desired behavior.

A set of 79 limit cycles was extracted with speeds varying from 0.42 m/s to 0.81 m/s. They were indexed by $p\in\mathcal{P}$ in an increasing order of their speed. The maximum speed gap between any two consecutively indexed limit cycles was 0.01 m/s. If such accuracy is not desired, the number of limit cycles can be reduced. For this set of limit cycles, $\zeta_{\rm ub}^* = 247.2~\mathrm{(kg m^2/s)^2}$, $\zeta_{\rm lb}^* = 120.8~\mathrm{(kg m^2/s)^2}$, and $K/\delta_{\rm z}^2 = 90.3~\mathrm{(kg m^2/s)^2}$, thus $\zeta_{\rm lb}^* \geq K/\delta_{\rm z}^2$ satisfying Theorem~\ref{thm:switching}. Hence, $\zeta[k]\in [120.8,247.2]$ for all $k\in\mathbb{Z}_+$.

The directed graph---as shown in Fig.~\ref{fig:digraph}---is constructed by running numerical simulations wherein we switch from each fixed point to every other and check the modeling constraints. The dwell-time used as a weight on the edges of $\mathcal{G}$ is computed using \eqref{eq:dwell-time} with $\epsilon=2$. The graph is strongly connected, i.e., every limit cycle can be reached from every other limit cycle, but it may require multiple switches. It is observed from the graph that speeding up usually does not require a long sequence of switches, however, slowing down does. This is illustrated in Fig.~\ref{fig:digraph} by marking out the 12 limit cycle switches required to go from 0.81 m/s to 0.42 m/s. On the contrary, to go from 0.42 m/s to 0.81 m/s a single switch is sufficient.

For the purpose of simulation consider the scenario when the biped is assumed to start on the limit cycle corresponding to 0.81 m/s. The desired speed is reduced to 0.42 m/s and, when convergence is achieved, it is changed back to 0.81 m/s. The speed convergence of the biped can be seen in Fig.~\ref{fig:speed_change}. It takes about 70 s for the biped to go from 0.81 m/s to 0.42 m/s while only 12 s to go from 0.42 m/s to 0.81 m/s. This disparity in convergence times occur due to the requirement of more switches for deceleration than acceleration as discussed earlier. All the modeling constraints were satisfied during the transients as well.

%%fig
%\begin{figure*}[t]
%\centering
%\subfigure[]
%{
%\includegraphics[width=0.31\textwidth]{figureEPS/continuum_3D.eps}
%\label{fig:continuum}
%}
%\centering
%\subfigure[]
%{
%\includegraphics[width=0.30\textwidth]{figureEPS/digraph.eps}
%\label{fig:digraph}
%}
%\centering
%\hspace*{-1mm}
%\subfigure[]
%{
%\includegraphics[width=0.33\textwidth]{figureEPS/speed_change_check.eps}
%\label{fig:speed_change}
%}
%\vskip -10pt
%\caption{\textbf{(a)} Continuum of limit cycles generated using the controller design presented in Section~\ref{sec:mod_con}. The average speed of the limit cycles increase from green to red. Slowest limit cycle is at 0.42 m/s while the fastest is at 0.81 m/s. \textbf{(b)} Directed graph of allowable transitions between limit cycles of different speeds. The path from 0.81 m/s (yellow ball) to 0.42 m/s (green ball) is depicted by red arrows. The tail of the arrow is on the source node and the head is on the destination node. \textbf{(c)} Simulation results for speed change of the biped from 0.81 m/s to 0.42 m/s and then back to 0.81 m/s. Average speed of the biped is blue while the desired speed is red.}
%\vskip -10pt
%\end{figure*}
%%fig

%fig
\begin{figure*}[t]
\centering
\subfigure[]
{
\includegraphics[width=0.41\textwidth]{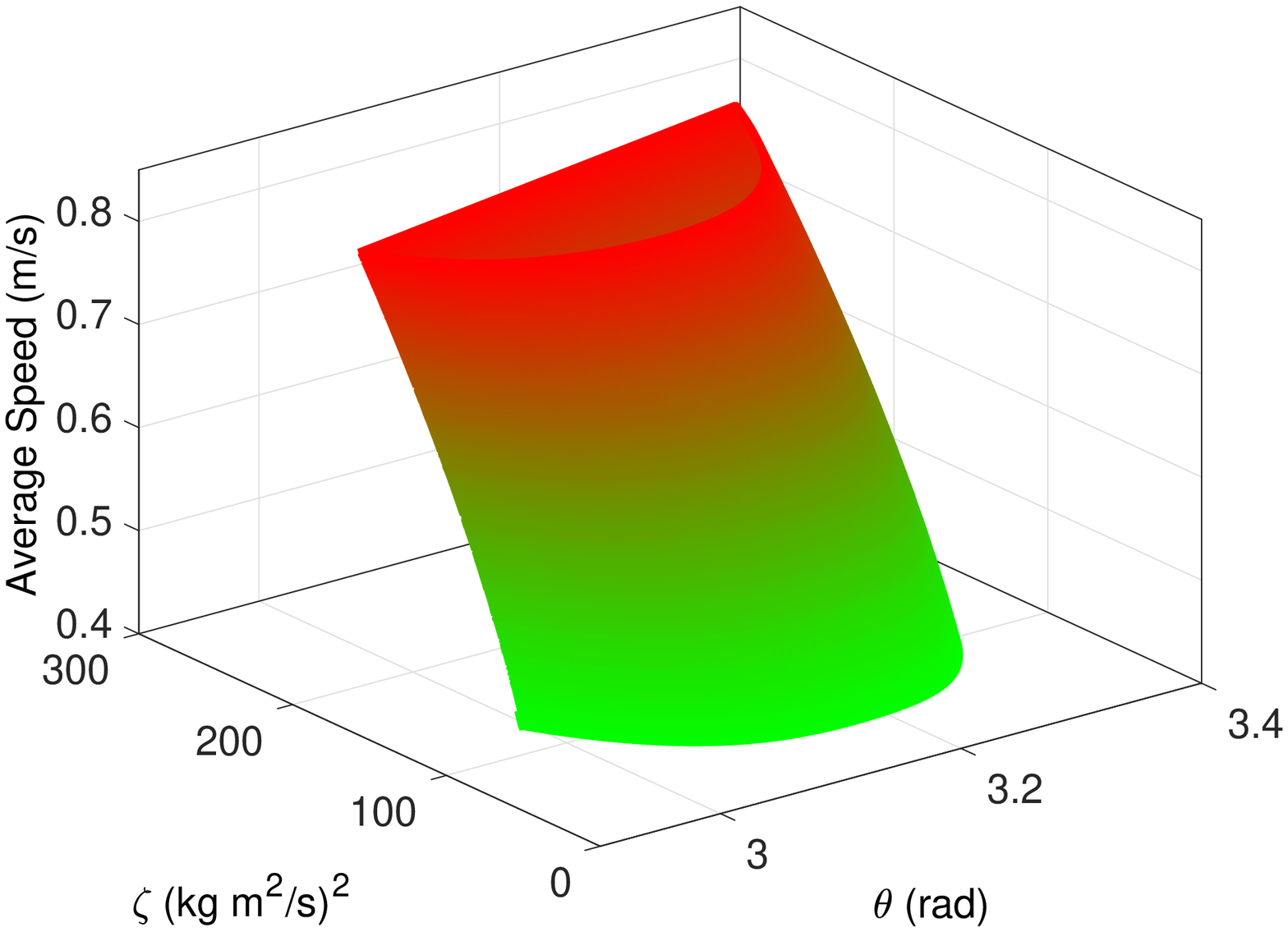}
\label{fig:continuum}
}
\centering
\hspace{10mm}
\subfigure[]
{
\includegraphics[width=0.41\textwidth]{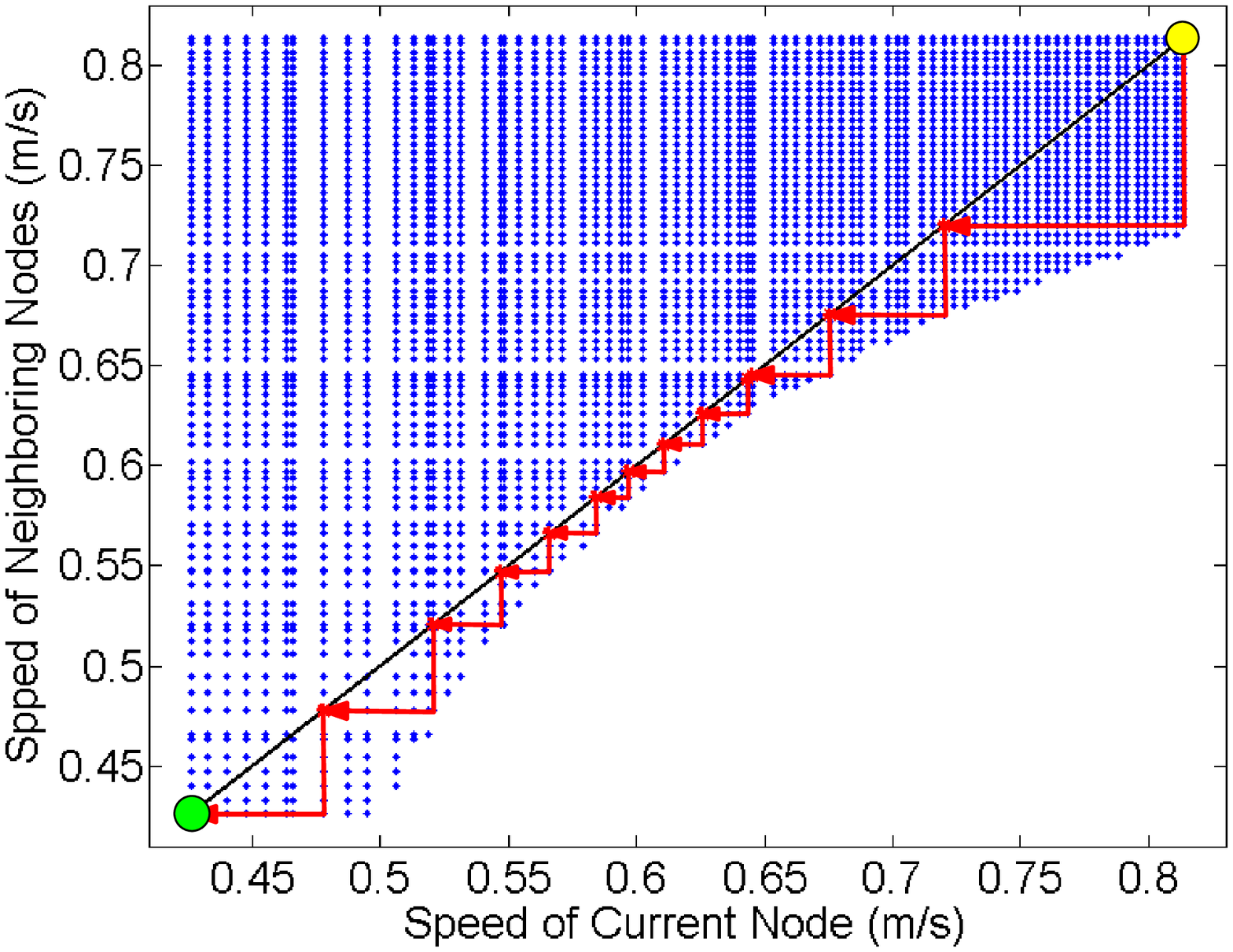}
\label{fig:digraph}
}
\vskip -10pt
\caption{\textbf{(a)} Continuum of limit cycles generated using the controller design presented in Section~\ref{sec:mod_con}. The average speed of the limit cycles increase from green to red. Slowest limit cycle is at 0.42 m/s while the fastest is at 0.81 m/s. \textbf{(b)} Directed graph of allowable transitions between limit cycles of different speeds. The path from 0.81 m/s (yellow ball) to 0.42 m/s (green ball) is depicted by red arrows. The tail of the arrow is on the source node and the head is on the destination node.}
\vskip -10pt
\end{figure*}
%fig

%fig
\begin{figure}[t]
%\label{fig:cycle_posF}
\centering
\includegraphics[width=0.37\textwidth]{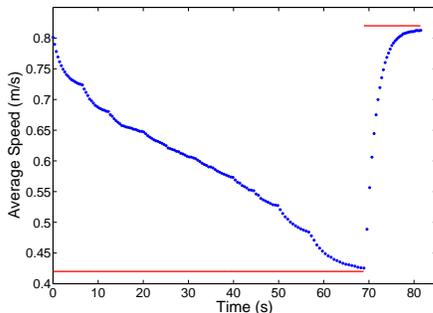}
%\centering
%\hspace{-0mm}
%\subfigure[]
%{
%\includegraphics[width=0.40\textwidth]{figureEPS/u.eps}
%\label{fig:u}
%}
\vskip -10pt
\caption{Simulation results for speed change of the biped from 0.81 m/s to 0.42 m/s and then back to 0.81 m/s. Average speed of the biped is blue while the desired speed is red.\label{fig:speed_change}
}
\vskip -10pt
\end{figure}

%=================================================================
%=================================================================
\section{Conclusion}
\label{sec:conclusion}
%=================================================================
%=================================================================

This paper presents a method to generate a continuum of exponentially stable limit cycles from a single HZD based limit cycle. We provide analytical guarantees for boundedness of the state under arbitrary switching among the limit cycles and ensure satisfaction of the modeling constraints. The latter is achieved by building a graph of feasible limit cycle switches and enforcing a dwell-time constraint on the switching signal. The method is applied for speed planning of a planar bipedal robot, allowing for speeds anywhere between 0.42 m/s to 0.81 m/s. The goal of this work is to enable provably stable speed planning by switching among limit-cycle gaits of underactuated bipedal walkers, so that their range of behaviors is extended.

\bibliographystyle{IEEEtran}
\bibliography{speed_change}

\end{document}